\documentclass[a4paper, twoside, 11pt]{scrartcl}
  \usepackage{amsmath, amsfonts, amssymb, amsthm, graphicx, mathrsfs, bm,mathtools}

\newcommand{\Titolo}{Holomorphic curves in Mixed Shimura Varieties}

\title{\Titolo}
\author{Michele~Giacomini}
\date{}

\usepackage[utf8]{inputenc}
\usepackage[T1]{fontenc}
\usepackage{lmodern}
\usepackage[english]{babel}

\usepackage{graphicx}
\graphicspath{{./Figure/}}

\usepackage[nottoc]{tocbibind}

\usepackage[pdftex,
 colorlinks=true, pdfstartview=FitV, linkcolor=blue, citecolor=blue, urlcolor=blue,p lainpages=false,bookmarks,final]{hyperref}
\hypersetup{
  bookmarksopen, %
  bookmarksopenlevel=2, %
  pdftitle={\Titolo},
  pdfauthor={Michele Giacomini},
  }

\usepackage[backend=biber,%
style=alphabetic,%
sorting=nyt,maxnames=5,maxalphanames=5,%
giveninits=true,    %
isbn=false, doi=false, url=false, eprint=false %
]{biblatex}

\AtEveryBibitem{%
\clearfield{edition}
\clearlist{location}
} %

\DefineBibliographyStrings{english}{%
  bibliography = {References},
} %

\usepackage{comment}
\usepackage{xspace}
\usepackage{xfrac}
\usepackage[usenames,dvipsnames]{color}
\usepackage{enumerate}%
\usepackage{csquotes}
\usepackage{enumitem}

\usepackage{tikz}
\usetikzlibrary{bending}
\usetikzlibrary{arrows}
\usetikzlibrary{arrows.meta}

\renewcommand{\bar}[1]{\overline{#1}}
\renewcommand{\tilde}[1]{\widetilde{#1}}

\newcommand{\abs}[1]{\left\vert #1\right\vert}

\DeclareMathOperator{\Zar}{Zar}

\DeclareMathOperator{\unif}{unif}

\theoremstyle{definition}
\newtheorem{definition}{Definition}[section]
\theoremstyle{plain}
\newtheorem{theorem}[definition]{Theorem}

\newtheorem{lemma}[definition]{Lemma}

\theoremstyle{remark}
\newtheorem{remark}[definition]{Remark}

\numberwithin{equation}{section}

\begin{document}

\title{Holomorphic curves in Mixed Shimura Varieties}

\maketitle

\begin{abstract}
    \textbf{Abstract.} Following work by Ullmo and Yafaev, we propose and prove an analogue of the Bloch-Ochiai theorem in the context of mixed Shimura varieties. We follow the strategy and use results of previous articles by Ullmo and Yafaev and the author plus consequences of the Ax-Lindemann-Weierstrass theorem for mixed Shimura varieties due to Gao.
\end{abstract}

\section{Introduction} %
\label{sec:introduction}

    The Bloch-Ochiai theorem \cite[cf. Chapter 9, 3.9.19]{Kobayashi98} states that the Zariski closure of an holomorphic curve in an abelian variety is a coset of an abelian subvariety.

    \begin{theorem}[Bloch-Ochiai]
        \label{thm:Bloch-Ochiai}
        Let $A$ be an abelian variety and $f:\mathbb{C}\to A$ a non-constant holomorphic map. Then the Zariski closure of $f(\mathbb{C})$ is a translate of an abelian subvariety.
    \end{theorem}

    In \cite{UllmoYafaev18}, Ullmo and Yafaev formulate and prove an analogue of this result for compact Shimura varieties, later generalised to all Shimura varieties by the author.

    Let $\mathcal{D}$ be a hermitian symmetric space realised as a bounded symmetric domain in $\mathbb{C}^{n}$ via the Harish-Chandra embedding\cite[see Chapter 4]{Mok89}, $G$ its isometry group and $\Gamma\subset G(\mathbb{R})$ an arithmetic lattice. Let $S=\Gamma\backslash\mathcal{D}$. Assume that $S$ is a component of a Shimura variety; in particular $G$ is defined over $\mathbb{Q}$ and $\Gamma$ is a congruence subgroup of $G(\mathbb{Q})$. Finally consider a holomorphic function $f:\mathbb{C}^{m}\to \mathbb{C}^{n}$ such that $f(\mathbb{C}^{m})\cap \mathcal{D}\neq \emptyset$. 

    \begin{theorem}[{\cite[Theorem 1.2]{UllmoYafaev18} \cite[Theorem 1.4]{Giacomini18}}]
        \label{thm:holomorphic curves in compact Shimura}
        Let $\pi:\mathcal{D}\to S$ be the quotient map, $f$ as above and $V=f(\mathbb{C}^{m})\cap \mathcal{D}$. Assume $S$ is compact then the components of the Zariski closure $\Zar(\pi(V))$ of $\pi(V)$ in $S$ are weakly special subvarieties of  $S$.
    \end{theorem}

    For general definitions about Shimura varieties and weakly special subvarieties see \cite{UllmoYafaev14} and the references therein.

    Along with the Bloch-Ochiai theorem the above result draws inspiration from the hyperbolic Ax-Lindemann theorem, first proven by Ullmo and Yafaev in \cite{UllmoYafaev14} for compact Shimura varieties and then in general by Klingler, Ullmo and Yafaev in \cite{KlinglerUllmoYafaev16}.

    \begin{theorem}[Ax-Lindemann]
        \label{thm:Ax-Lindemann}
        Let $Y\subset \mathcal{D}$ be an algebraic subset of $\mathcal{D}$\footnote{An algebraic subset of $\mathcal{D}$ is a component of the intersection of an algebraic subset of $\mathbb{C}^{n}$ with $\mathcal{D}$}. Then the components of the Zariski closure $\Zar(\pi(Y))$ are weakly special.
    \end{theorem}

    The aim of this paper is to prove a generalisation of theorem \ref{thm:holomorphic curves in compact Shimura} to the case of mixed Shimura varieties.

    Let $ (P,\mathcal{X}^{+}) $ be a connected mixed Shimura datum, $ G $ the quotient of $ P $ by its unipotent radical and $ (G,\mathcal{X}^{+}_{G}) $ the quotient connected pure Shimura datum. Fix an arithmetic subgroup $ \Gamma \subset P(\mathbb{Q}) $ and denote by $ \Gamma_{G} $ the image of $ \Gamma $ in $ G $. Finally let $ \mathbb{C}^{N} $ be the holomorphic tangent space to $ \mathcal{X}^{+}_{G} $ at some point $ x_{0} $. Use the Harish-Chandra embedding theorem to embed $ \mathbb{C}^{N} $ as a Zariski open subset of $ \mathcal{X}^{+\vee}_{G} $ and embed $ \mathcal{X}^{+}_{G} $ as a analytically open and bounded subset of $ \mathbb{C}^{N} $. All the above fit in the following diagram.

    \begin{center}
            \begin{tikzpicture}[scale=1.5]
                    \node (Xv)      at      (0,1)       {$ \mathcal{X}^{+\vee} $};
                    \node (XGv)     at      (0,0)       {$ \mathcal{X}^{+\vee}_{G} $};
                    \node (CN)      at      (1,0)       {$ \mathbb{C}^{N} $};
                    \node (X)       at      (2,1)       {$ \mathcal{X}^{+} $};
                    \node (XG)      at      (2,0)       {$ \mathcal{X}^{+}_{G} $};
                    \node (S)       at      (5,1)       {$ S = \Gamma\backslash \mathcal{X}^{+} $};
                    \node (SG)      at      (5,0)       {$ S_{G} = \Gamma\backslash \mathcal{X}^{+}_{G}$};
                    \draw[-stealth] (Xv)        edge    node[left]  {$ \pi $}       (XGv)
                    (X)     edge    node[left]  {$ \pi $}   (XG)
                    (X)     edge    node[above] {$ \unif $} (S)
                    (XG)    edge    node[above] {$ \unif_{G} $} (SG)
                    (S)     edge    node[left]  {$ [\pi] $} (SG);
                    \draw[left hook - stealth] (X)      edge    (Xv)
                    (XG)        edge    (CN)
                    (CN)        edge    (XGv);
            \end{tikzpicture}
    \end{center}

    Now consider an holomorphic map $ f: \mathbb{C}\to \mathcal{X}^{+\vee} $ such that the image of the composition $ \pi\circ f $ is contained in $ \mathbb{C}^{N} $. We prove the following result.

    \begin{theorem}[Main result]
        \label{thm:holomorphic curves in mixed Shimura}
        Assume that $ unif_{G}(\pi\circ f(\mathbb{C})\cap \mathcal{X}^{+}_{G}) $ is not a single point. Then the Zariski closure $ Zar(unif(f(\mathbb{C})\cap \mathcal{X}^{+})) $ is a weakly special subvariety of $ S $.
    \end{theorem}

   \begin{remark}
       Note that if one wanted to generalise the above theorem without the requirement that the projection of the image of $ f $ to the pure part be non constant, one would need a generalisation of the Bloch-Ochiai theorem to the fibre of a general mixed Shimura variety. To the best of the author's knowledge this is not yet know; the most general case in this direction seems to be the Bloch-Ochiai theorem for semiabelian varieties \cite{Noguchi81}.
   \end{remark}

   The proof follows the general structure of the proof in the pure case. Using the Ax-Lindemann-Weierstrass theorem, we will first reduce the statement to the existence of a particular semialgebraic subset of $ P(\mathbb{R})U(\mathbb{C}) $. Then we will prove the existence of such a set using Pila-Wilkie's theorem on rational points in definable sets. This involves proving a counting result about the number of translates of a fixed fundamental domain intersecting the image of $f$. This counting result will follow from the pure case and our assumptions on $f$. 

   We remark that the assumption that $f$ be transverse to the fibres of the projection to the pure part is essential for us to be able to apply the present methods. Indeed, as remarked above, dropping this hypothesis would require us to prove also the Bloch-Ochiai theorem. This does not seem to be possible using the approach presented in this paper. Indeed, related questions were investigated with similar methods by Ullmo and Yafaev in \cite{UllmoYafaev17}; however the authors were able to prove only weaker results. This is ultimately due to the fact that the space uniformising a (semi-)abelian variety does not have any bounded realisation.

   \paragraph{Acknowledgements.}
      \label{par:Acknowledgements.}
      I would like to thank my supervisor Andrei Yafaev for pointing me out the problem and Ziyang Gao for the helpful discussions during the preparation of this paper. I would also like to thank the referee for the helpful comments and a remark on the proof (cf. remark \ref{rem:Pila-Wilkie}).

      \noindent This work was supported by the Engineering and Physical Sciences Research Council [EP/L015234/1]. 
      The EPSRC Centre for Doctoral Training in Geometry and Number Theory (The London School of Geometry and Number Theory), University College London.

\section{Preliminaries} %
\label{sec:preliminaries}

    \subsection{Mixed Shimura varieties} %
    \label{sub:mixed_shimura_varieties}

        Here we briefly recall the definition of mixed Shimura variety. For more details see \cite{Pink90}.

        \begin{definition}
            \label{def:mixed datum}
            A \emph{connected mixed Shimura datum} is a triple $ (P,\mathcal{X}^{+},h) $ where
            \begin{itemize}
                \item $ P $ is a linear algebraic group defined over $ \mathbb{Q} $, with an algebraic subgroup $ U $ of its unipotent radical $ W $ uniquely determined by condition \ref{bullet: condition for U} below;
                \item $ \mathcal{X}^{+} $ is a left homogeneous space under the group $ P(\mathbb{R})^{+}U(\mathbb{C}) $;
                \item $ h:\mathcal{X}^{+}\to \mathrm{Hom}(\mathbb{S}_{\mathbb{C}},P_{\mathbb{C}}) $ is a $ P(\mathbb{R})U(\mathbb{C}) $-equivariant map;
            \end{itemize}
            such that for all $ x\in \mathcal{X}^{+} $ the following axioms are true:
            \begin{enumerate}
                \item every fibre of $ h $ consists of finitely many points;
                \item let $ \pi':P\to P/U $ be the canonical projection, then $ \pi'\circ h_{x}:\mathbb{S}_{\mathbb{C}}\to P/U_{\mathbb{C}} $ is already defined over $ \mathbb{R} $;
                \item let $ \pi:P\to G=P/W $ be the canonical projection and $ w:\mathbb{G}_{m}\to \mathbb{S} $ be the map which on real points is the inclusion of $ \mathbb{R}^{*} $ into $ \mathbb{C}^{*} $, then $ \pi\circ h_{x}\circ w: \mathbb{G}_{m,\mathbb{R}}\to G_{\mathbb{R}} $ is a cocharacter of the centre of $ G $;
                \item $ Ad\circ h_{x} $ induces on $ \mathrm{Lie}(P) $ a rational mixed Hodge structure of type
                \begin{equation}
                    \left\{(1,-1),(0,0),(-1,1)  \right\}\cup \left\{ (-1,0),(0,1) \right\}\cup \left\{ (-1,-1) \right\};
                \end{equation}
                \item\label{bullet: condition for U} let $ V=W/U $ then the weight filtration on $ \mathrm{Lie}(P) $ is as follows
                \begin{equation}
                    W_{n}(\mathrm{Lie}(P))=\begin{cases}
                        0                   &\text{if } n < -2\\
                        \mathrm{Lie}(U)     &\text{if } n = -2\\
                        \mathrm{Lie}(V)     &\text{if } n = -1\\
                        \mathrm{Lie}(P)     &\text{if } n \geq 0;
                    \end{cases}
                \end{equation}
                \item $ int(\pi\circ h_{x}(-1)) $ induces a Cartan involution on $ G_{\mathbb{R}}^{ad} $;
                \item $ G^{ad} $ possesses no non trivial factor of compact type defined over $ \mathbb{Q} $;
                \item the centre of $ G $ decomposes as an almost direct product of a $ \mathbb{Q} $-split torus with a torus of compact type over $ \mathbb{Q} $.
            \end{enumerate}
        \end{definition}

        Below we will omit the subgroup $ U $ and the map $ h $ from the notation of Mixed Shimura data.

        \begin{remark}
            In the above definition it is sufficient to check that the conditions are satisfied for a single point $ x\in \mathcal{X}^{+} $. For an explanation of the conditions in the definition see \cite[Remark 2.2]{Gao17} and for more details \cite[Chapter 1]{Pink90}.
        \end{remark}

        \begin{definition}
            \label{def:mixed Shimura varieties}
            Let $ (P,\mathcal{X}) $ be a mixed Shimura datum and $ K $ a compact open subgroup of $ P(\mathbb{A}_{f}) $; where $ \mathbb{A}_{f} $ is the ring of finite adèles over $ \mathbb{Q} $. The \emph{mixed Shimura variety} associated to the datum $ (P,\mathcal{X}) $ and the subgroup $ K $ is
            \begin{equation}
                M_{K}(P,\mathcal{X}) = P(\mathbb{Q})\backslash \mathcal{X}\times P(\mathbb{A}_{f})/K,
            \end{equation}
            where $ P(\mathbb{Q}) $ acts on the left diagonally on both factors and $ K $ acts on the second factor on the right.
        \end{definition}

        \begin{remark}
            From \cite[3.2]{Pink90} we have that connected mixed Shimura varieties defined as above are connected components of mixed Shimura varieties as defined usually.
        \end{remark}

    \subsection{Weakly special subvarieties in Mixed Shimura varieties} %
    \label{sub:weakly_special_subvarieties_in_mixed_shimura_varieties}

        Now we briefly recall the definition of weakly special subvariety of a mixed Shimura variety, the mixed Ax-Lindemann theorem and one of its consequences. For more details see \cite{Gao17}.

        \begin{definition}
            \label{def:weakly special subvarieties}
            Let $ (P,\mathcal{X}^{+}) $ be a connected mixed Shimura datum. A subset $ U $ of $ \mathcal{X}^{+} $ is said to be \emph{weakly special} if there is a diagram of Shimura morphisms
            \begin{center}
                    \begin{tikzpicture}[scale=1.2]
                            \node (X)       at      (2,0)       {$ (P,\mathcal{X}^{+}) $};
                            \node (Y)       at      (1,1.2)     {$ (Q,\mathcal{Y}^{+}) $};
                            \node (Y')      at      (0,0)       {$ (Q',\mathcal{Y}'^{+}) $};
                            \draw[-latex]   (Y)     edge    node[right] {$ i $}     (X)
                                            (Y)     edge    node[left]  {$ \varphi $}   (Y');
                    \end{tikzpicture}
                \end{center}
            and a point $ y'\in \mathcal{Y}' $ such that $ U $ is a connected component of $ i(\varphi^{-1}(y')) $.
        \end{definition}

        \begin{remark}
            In the above definition the morphisms $\varphi$ and $ i $ may be chosen to be surjective and injective respectively \cite[cf.]{Gao17}
        \end{remark}

        \begin{theorem}[{Mixed Ax-Lindemann \cite[Theorem 1.2]{Gao17}}]
            \label{thm:mixed A-L}
            Let $ S $ be a connected mixed Shimura variety and $ (P,\mathcal{X}^{+}) $ the associated mixed Shimura datum. Let $ \unif:\mathcal{X}^{+}\to S $  be the uniformisation map. Let $ Y $ be a closed irreducible subvariety of $ S $. Let $ \tilde{Y} $ be an irreducible algebraic subset contained in $ \unif^{-1}(Y) $ and maximal for these properties. Then $ \tilde{Y} $ is weakly special.
        \end{theorem}

        The following result is a consequence of the mixed Ax-Lindemann theorem above, it is the analogue of \cite[Théorème 1.3]{Ullmo14} in the mixed case.

        \begin{theorem}[{\cite[Theorem 12.2]{Gao17}}]
            \label{thm:subvars with zar-dense w.s. subvars}
            Let $ S $ be a connected mixed Shimura variety associated with the datum $ (P, \mathcal{X}^{+}) $. Let $ Y $ be an irreducible Hodge generic algebraic subvariety of $ S $. Then there exists a normal subgroup $ N\triangleleft P $ with the following property. Construct the diagram
            \begin{center}
                    \begin{tikzpicture}[scale=1.3]
                            \node (P)       at      (0,1)       {$(P,\mathcal{X}^{+})$};
                            \node (P/N)     at      (3,1)       {$(Q,\mathcal{Y}^{+})=(P,\mathcal{X})/N$};
                            \node (S)       at      (0,0)       {$ S $};
                            \node (SN)      at      (3,0)       {$ S_{P/N} $};
                            \draw[-latex]       (P)         edge    node[above] {$ \rho $}      (P/N)
                                                (P)         edge    node[left]  {$ \unif $}     (S)
                                                (P/N)       edge    node[left]  {$ \unif_{N} $} (SN)
                                                (S)         edge    node[above] {$ [\rho] $}    (SN);
                    \end{tikzpicture}
            \end{center}
            where the Shimura morphism $ \rho $ corresponds to the projection map $ P\to P/N $. Then the union of weakly special subvarieties of $ S_{N} $ contained in $ Y' = \Zar([\rho](Y)) $ is not Zariski dense in $ Y' $ and $ [\rho]^{-1}(Y')=Y $.
        \end{theorem}

\section{Proof of the main result} %
\label{sec:proof_of_the_main_result}
    
    \subsection{A first reduction}
    \label{sub:a_first_reduction}

        Recall that $ \pi:\mathcal{X}^{+}\to \mathcal{X}^{+}_{G} $ is an holomorphic vector bundle. Let $ x_{0}\in \partial \mathcal{X}^{+}_{G}\cap \pi\circ f(\mathbb{C}) $ be some fixed point. Choose a sufficiently small number $ R>0 $ such that the restriction of $ \mathcal{X}^{+} $ to the intersection of the open ball $ B_{\mathbb{C}^{N}}(x_{0},R) $ in $ \mathbb{C}^{N} $ with $ \mathcal{X}^{+}_{G} $ is the trivial vector bundle. Let $ A'= (\pi \circ f)^{-1}(\pi\circ f(\mathbb{C})\cap \mathcal{X}^{+}_{G}\cap B_{\mathbb{C}^{N}}(x_{0},R))$. Then $ A' $ is an open subset of $ \mathbb{C} $ and, by definition there exists some $ R'>0 $ such that $ \bar{f(A'\cap B_{\mathbb{C}}(0,R'))}\cap \partial \mathcal{X}^{+}_{G} \neq \emptyset $. Let $ A = A'\cap B(0,R') $. Finally let $ Z = f(A) $, $ Y = \Zar(\unif(Z)) $ and $ \tilde{Y} $ be the analytic component of $ \unif^{-1}(Y) $ containing $ Z $.

        Now we note that we may reduce to the case $ Y $ is Hodge generic. Moreover we may assume that $ \Gamma $ is neat and moreover that we may decompose $ \Gamma $ as a semidirect product $ \Gamma = \Gamma_{G} \ltimes \Gamma_{W}$\cite[cf. 4.1, Corollary 2]{PlatonovRapinchuk94}.

        \begin{remark}
            By analytic continuation, the Zariski closure of $ Z = f(A) $ in $ \mathcal{X}^{+} $ is the same as the Zariski closure of $ f(\mathbb{C}) $. Hence it is sufficient to prove that $ Y $ is weakly special.
        \end{remark}

        We will deduce theorem  \ref{thm:holomorphic curves in mixed Shimura} from the following result.

        \begin{theorem}
            \label{thm:semialgebraic set}
            In the hypotheses of theorem \ref{thm:holomorphic curves in mixed Shimura} and notation as above, there exists a positive dimensional semialgebraic set $ X\subset P(\mathbb{R})^{+}U(\mathbb{C}) $, that is not contained in the stabilizer of any point, such that $ X.Z \subset \tilde{Y} $.
        \end{theorem}

        Now we show how to deduce theorem \ref{thm:holomorphic curves in mixed Shimura} from the above result using the Ax-Lindemann-Weierstrass theorem.

        \begin{proof}[Proof of theorem \ref{thm:holomorphic curves in mixed Shimura}]
            Let $ z \in Z $ be any point, then by theorem \ref{thm:semialgebraic set}, the maximal semialgebraic subset $ X $ of $ \tilde{Y} $ containing $ z $ has positive dimension. By \cite[Lemma 4.1]{PilaTsimerman13}, $ X $ is a complex algebraic subset of $ \tilde{Y} $ and by the Ax-Lindemann-Weierstrass theorem for mixed Shimura varieties \cite[cf. Theorem 1.2]{Gao17} it is a weakly special subset of $ \mathcal{X}^{+} $ contained in $ \tilde{Y} $. This implies that $ Z $ is covered by positive dimensional weakly special subsets of $ \mathcal{X}^{+} $ contained in $ \tilde{Y} $. Hence the set of positive dimensional weakly special subvarieties of $ S $ contained in $ Y $ is Zariski dense in $ Y $.

            Since by assumption $ Y $ is Hodge generic, we may apply \cite[Theorem 12.2]{Gao17} and obtain a normal subgroup $ N $ of $ P $ with the following properties. Let $ \rho: P\to P/N $ denote the quotient map and let $ [\rho]:S\to S_{P/N} $ be the associated map on Shimura varieties, let $ Y'=[\rho](Y) $. Then the set of weakly special subvarieties of $ S_{P/N} $ contained in $ Y' $ is not Zariski dense in $ Y' $ and $ Y = [\rho]^{-1}(Y') $.
            Since $ Y' = [\rho](Y) $ and $ Y = [\rho]^{-1}(Y') $, we have that $ [\rho](\unif(Z)) $ is Zariski dense in $ Y' $. Hence, if $ Y' $ had dimension bigger than zero, the composition $ \rho\circ f: \mathbb{C}\to \mathcal{X}_{P/N}^{+\vee} $ would be non constant and, applying the above reasoning to it, we would get a set of weakly special subvarieties of $ S_{P/N} $ contained in $ Y' $ and Zariski dense in it. In this way we get that $ Y' $ has dimension zero, hence it is a point and, by definition of weakly special subvariety, $ Y $ is weakly special.
        \end{proof}

    \subsection{Proof of theorem \ref{thm:semialgebraic set}} %
    \label{sec:proof_of_theorem_thm:semialgebraic set}
        
        We will use Pila-Wilkie's theorem on rational points in definable sets to complete the proof of the result. This will involve proving a counting result about the number of fundamental domains intrsecting the image of the function $f$. The counting result will be deduced from the case of pure Shimura varieties using the fact that the composition of $f$ with the projection to the pure part is not constant. 

        The following lemma will allows us to reduce the counting to the pure case.

        \begin{lemma}
            \label{lemma:finitely many translates containing V}
            There is a finite subset $ \Lambda $ of $ \Gamma_{W} $ such that $ Z\subset \Lambda\Gamma_{G}.\mathscr{F} $. 
        \end{lemma}
        \begin{proof}

            Note that we can choose the trivialisation of the vector bundle $ \mathcal{X}\to \mathcal{X}_{G} $ above $ Z $ as $ Z\times W(\mathbb{R})U(\mathbb{C}) $, so that the action of $ W(\mathbb{R})U(\mathbb{C}) $ is exactly the action on the second component. Now, since $ \bar{\pi(Z)}\cap \partial \mathcal{X}_{G} \!= \emptyset$, we have that the projection of $ Z $ to $ W(\mathbb{R})U(\mathbb{C}) $ is bounded. Hence there is a finite subset $ \Lambda $ of $ \Gamma_{W} $ such that the translates of a (fixed) fundamental set for $ \Gamma_{W} $ in $ W(\mathbb{R})U(\mathbb{C}) $ cover the projection of $ Z $ to $ W(\mathbb{R})U(\mathbb{C}) $. This proves the lemma.
        \end{proof}

        Now we may assume there is a faithful finite dimensional representation $\rho:P\to \mathrm{GL}(E)$ of $P$ defined over $\mathbb{Q}$ and some lattice $E_{\mathbb{Z}}$ such that $\Gamma=G(\mathbb{Z})=G(\mathbb{Q})\cap \mathrm{GL}(E_{\mathbb{Z}})$. With this assumption, we can give the following definition.
         
        \begin{definition}
            \label{def:Height}
            Let $\rho:P\to GL(E)$ be the faithful finite dimensional representation of $P$ fixed above; for any $\gamma\in \Gamma$ write $\rho(\gamma)=(\gamma_{i,j})_{i,j}$. For any $\phi\in End(E_{\mathbb{R}})$ define
            \begin{equation}
                \abs{\phi}_{\infty}=\max_{i,j}\abs{\phi_{i,j}}.
            \end{equation}
            Moreover, define the \emph{height} of $\gamma\in \Gamma$ as 
            \begin{equation*}
                H(\gamma)=\max(1,\abs{\gamma}_{\infty}).
            \end{equation*}
        \end{definition}

        \begin{remark}
            \label{rem:height of product}
            Let $n$ be the dimension of $E$, then for any $\gamma_{1},\gamma_{2}\in \Gamma$, we have
            \begin{equation}
                H(\gamma_{1}\gamma_{2})\leq n H(\gamma_{1})H(\gamma_{2}).
            \end{equation}
        \end{remark}

        \begin{definition}
            \label{def:NVT}
            \begin{equation}
                N_{Z}(T) = \# \left\{ \gamma\in \Gamma | \gamma.\mathscr{F}\cap Z \neq \emptyset \text{ and } H(\gamma)\leq T \right\}.
            \end{equation}
        \end{definition}

        \begin{theorem}
            \label{thm:NVT large}
            There exist constants $c_{1},c_{2}>0$ such that for all $T>0$ sufficiently large
            \begin{equation}
                N_{Z}(T)\geq c_{1}T^{c_2}.
            \end{equation}
        \end{theorem}
        \begin{proof}
            Consider the two sets
            \begin{equation}
                \begin{split}
                \Sigma_{Z}     &= \left\{ \gamma\in \Gamma | \gamma.\mathscr{F}\cap Z \neq \emptyset \right\},\\
                \Sigma_{Z}^{G} &= \left\{ \gamma\in \Gamma_G | \pi(\gamma).\pi(\mathscr{F})\cap \pi(Z) \neq \emptyset \right\}.
                \end{split}
            \end{equation}
            By lemma \ref{lemma:finitely many translates containing V} and remark \ref{rem:height of product}, we know that there is a number $k$ depending only on $f$, the mixed Shimura datum $(P,\mathcal{X})^+$, the subgroup $\Gamma$ and the choice of fundamental domain $\mathscr{F}$ such that if $\gamma_{G}\in \Sigma_{Z}^{G}$ is such that $H(\gamma_{g})\leq N$, then there is some $\gamma\in \Sigma_{V}$ such that $\pi(\gamma)=\gamma_{G}$ and $H(\gamma)\leq k N$. Now it is sufficient to apply \cite[Theorem 3.2]{Giacomini18} to get the desired result. 
        \end{proof}

        We can now complete the proof of theorem \ref{thm:semialgebraic set} using the Pila-Wilkie theorem. To do this we need a last lemma.

        \begin{lemma}
            \label{lemma:Sigma}
            Consider the set 
            \begin{equation}
                \Sigma(Y)=\left\{ p\in P(\mathbb{R})^+ U(\mathbb{C}) \mathop{|} \dim(p.Z\cap \mathcal{F}\cap \pi^{-1}(Y))= \dim (Z) \right\}.
            \end{equation}
            Then the set $\Sigma(Y)$ is definable in $\mathbb{R}_{an,exp}$. For all $p\in \Sigma(Y)$, $p.Z\subseteq \pi^{-1}(Y)$. Moreover, define
            \begin{equation}
                \Sigma'(Y)=\left\{ p\in  P(\mathbb{R})^+ U(\mathbb{C}) \mathop{|} Z\cap p^{-1}.\mathcal{F}\neq \emptyset \right\}.
            \end{equation}
            Then
            \begin{equation}
                \Sigma(Y)\cap \Gamma= \Sigma'(Y)\cap \Gamma.
            \end{equation}
        \end{lemma}
        \begin{proof}
            The set $\Sigma(Y)$ is definable in $\mathbb{R}_{an,exp}$ because all sets and maps involved in its definition are\footnote{For the definablility of the uniformisation map see \cite[Section 10]{Gao17}.}. The second assertion follows by analytic continuation. Finally the equality
            \begin{equation}
                \Sigma(Y)\cap \Gamma= \Sigma'(Y)\cap \Gamma
            \end{equation}
            follows form the fact that $\pi^{-1}(Y)$ is $\Gamma$-invariant.
        \end{proof}

        We now recall a consequence of the Pila-Wilkie counting theorem.

        \begin{theorem}
            \label{thm:Pila-Wilkie}
            Let $S\subset \mathbb{R}^{n}$ be a set definable in the o-minimal structure $\mathbb{R}_{an,exp}$. Denote by $N_{S,\mathbb{Z}}(T)$ the number of points $s=(s_{ 1},\ldots,s_{ n}) \in  S\cap \mathbb{Z}^{n} $ such that $ \max \abs{s_{ i}}\leq T $. Fix a natural number $ k $. If there exist constants $c,\varepsilon>0$ such that $N_{S,\mathbb{Z}}(T)>cT^{\varepsilon}$ for all $T$ sufficiently large, then $ S $ contains a positive dimensional semialgebraic set containing at least $ k $ points in $ S\cap \mathbb{Z}^{n} $.
        \end{theorem}
        
        \begin{remark}
        \label{rem:Pila-Wilkie}
            The above theorem follows from the version of the Pila-Wilkie theorem for semialgebraic blocks proven by Pila in \cite{Pila11}. The original version of the theorem proven in \cite{PilaWilkie06} is not strong enough for our purposes because it does not imply that there is a single semialgebraic set containing many rational or, in this case, integer points. We use the additional information to prove that the semialgebraic set we obtain does not stabilise any point. The author would like to thank the referee for pointing this out to him.
        \end{remark}

        Let 
        \begin{equation}
            N_{\Sigma(W)}(T)= \left\{ \gamma\in \Gamma \cap \Sigma(W) \mathop{|} H(\gamma)\leq T \right\}.
        \end{equation}
        From theorem \ref{thm:NVT large}, we see that $N_{\Sigma(Y)}(T)\geq c_{1}T^{c_{2}}$, for some constants $c_{1},c_{2}>0$. Combining this with theorem \ref{thm:Pila-Wilkie}, we get for any integer $ k $ a semialgebraic set $ X\subset \Sigma(Y) $ containing more than $ k $ points of $\Gamma$; choosing a big enough $ k $ we may assume that $X$ is not contained in the stabilizer of any point. Finally from lemma \ref{lemma:Sigma} we get that $X . Z\subset \pi^{-1}(Y) $, thus proving theorem \ref{thm:semialgebraic set}.

\printbibliography[heading=bibintoc]

\end{document}